\documentclass[12pt,a4paper]{amsart}
\usepackage{amsmath,amssymb}
\usepackage{bm}
\usepackage{enumitem}
\theoremstyle{plain}
\newtheorem{thm}{Theorem}[section]
\newtheorem{lem}[thm]{Lemma}
\newtheorem{prop}[thm]{Proposition}
\newtheorem{cor}[thm]{Corollary}
\newtheorem{fact}[thm]{Fact}

\theoremstyle{definition}

 \newtheorem{ques}{Question}
 \newtheorem{conj}{Conjecture}

\begin{document}
\setlength{\baselineskip}{16pt}
\title[Characterization of a metrizable space $X$ such that $F_4(X)$ is Fr\'echet-Urysohn]
       {Characterization of a metrizable space $X$ such that $F_4(X)$ is Fr\'echet-Urysohn}
\author[Kohzo Yamada]{Kohzo Yamada}
\address{Department of Mathematics, Faculty of Education, 
Shizuoka University, Shizuoka, 422-8529 Japan}
\email{kohzo.yamada@shizuoka.ac.jp}
\maketitle
%
%
\begin{abstract}
Let $F(X)$ be the free topological group on a Tychonoff 
space $X$. For all natural numbers $n$ we denote by $F_n(X)$ the 
subset of $F(X)$ consisting of all words of reduced length $\leq n$. 
In \cite{Y3}, the author found equivalent conditions on a metrizable 
space $X$ for $F_3(X)$ to be Fr\'echet-Urysohn, and for $F_n(X)$ 
to be Fr\'echet-Urysohn for $n\geq5$. However, no equivalent 
condition on $X$ for $n=4$ was found. In this paper, we give the 
equivalent condition. In fact, we show that for a metrizable space 
$X$, if the set of all non-isolated points of $X$ is compact, then 
$F_4(X)$ is Fr\'echet-Urysohn. Consequently, for a metrizable space $X$ $F_3(X)$ 
is Fr\'echet-Urysohn if and only if $F_4(X)$ is Fr\'echet-Urysohn. 
\end{abstract}
%
%
\section{Introduction}
Let $X$ be a Tychonoff space and $F(X)$ and $A(X)$ be respectively 
the \textit{free topological group}\/ and the \textit{free abelian 
topological group}\/ on $X$ in the sense of Markov \cite{M}. 
As an abstract group, $F(X)$ is free on $X$ and 
every continuous map from $X$ to an arbitrary topological group 
lifts in a unique fashion to a continuous homomorphism from $F(X)$. 
Similarly, as an abstract group, $A(X)$ is the free abelian group 
on $X$ and every continuous map from $X$ to an 
arbitrary abelian topological group extends to a unique continuous 
homomorphism from $A(X)$.
\par
For each $n\in{\mathbb N}$, $F_n(X)$ stands for the subset of 
$F(X)$ formed by all words of reduced length at most $n$. It is 
known that $X$ itself and each $F_n(X)$ is closed in $F(X)$. 
The subspace $A_n(X)$ is defined similarly and each $A_n(X)$ is 
closed in $A(X)$. 
Denote by $\widetilde{X}$ the topological sum 
of $X$, its copy $X^{-1}$, and $\{e\}$; that is, 
$\widetilde{X}=X\oplus\{e\}\oplus X^{-1}$, where $e$ is the unit element of $F(X)$. 
For every $n\in{\mathbb N}$, denote by $i_n$ the multiplication mapping of 
$\widetilde{X}\,\rule{0pt}{11pt}^n$ to $F(X)$; i.e, 
$i_n(x_1,x_2,\ldots,x_n)=x_1x_2\cdots x_n$ for each point 
$(x_1,x_2,\ldots,x_n)\in\widetilde{X}\,\rule{0pt}{11pt}^n$. 
Clearly, every $i_n$ is continuous and 
$i_n(\widetilde{X}\,\rule{0pt}{11pt}^n)=F_n(X)$. 
In dealing with $A(X)$ and $A_n(X)$, we use additive notation. So, the space 
$\widetilde{X}=X\oplus\{0\}\oplus-X$ and the mapping 
$i_n:\widetilde{X}\,\rule{0pt}{11pt}^n\to A_n(X)$ is defined by 
$i_n(x_1,x_2,\ldots,x_n)=x_1+x_2+\cdots+x_n$. Of course, $i_n$ is continuous 
and $i_n(\widetilde{X}\,\rule{0pt}{11pt}^n)=A_n(X)$.
\par 
Recall that a topological space $X$ is called 
\textit{Fr\'echet-Urysohn} if for every $A\subseteq X$ and 
every $x\in\overline{A}$ there is a sequence of points of $A$ 
converging to $x$. In \cite{Y2, Y3} the following results were obtained. 
\begin{prop}[{\cite[Proposition 4.8]{Y2}}] 
If a space $X$ is paracompact, then $i_2$ is a closed mapping. 
Therefore, for a metrizable space $X$ both $A_2(X)$ and $F_2(X)$ 
are Fr\'echet-Urysohn. 
\end{prop}
\begin{thm}[{\cite[Corollary 2.5 and Corollary 2.6]{Y3}}] 
Let $X$ be a metrizable space. Then{\em :} 
\begin{enumerate} [leftmargin=*,label={\em (\arabic*)},parsep=0pt, topsep=5pt] 
\item $F_n(X)$ is Fr\'echet-Urysohn for each natural number $n\geq5$ if 
and only if $X$ is compact or discrete.
\item The following are equivalent{\em:}
\begin{enumerate}
[leftmargin=*,label={\em (\roman*)},parsep=0pt, topsep=5pt]
\item $A_n(X)$ is Fr\'echet-Urysohn for each natural number $n\geq3${\em;}
\item $F_3(X)$ is Fr\'echet-Urysohn{\em;} 
\item the set of all non-isolated points of $X$ is compact. 
\end{enumerate}
\end{enumerate} 
\end{thm} 
The reader may notice that  an equivalent condition on $X$ for 
$F_4(X)$ to be Fr\'echet-Urysohn is not given. The author 
in \cite{Y3} posed the following conjecture. 
\begin{conj}
$F_4(X)$ is Fr\'echet-Urysohn\ if the set of all non-isolated points of a metrizable space $X$ is compact. 
\end{conj}
Recently, the following partial answers were obtained. 
\begin{thm}[{\cite[Corollary 3.7]{Y4}}]
Let $X$ be a locally compact separable metrizable space such that 
the set of all non-isolated points of $X$ is compact. Then $F_4(X)$ is 
Fr\'echet-Urysohn. 
\end{thm}
\begin{thm}[{\cite[Theorem 3.3]{Y5}}]
Let $X$ be a locally compact metrizable space such that 
the set of all non-isolated points of $X$ is compact. Then $F_4(X)$ is 
Fr\'echet-Urysohn. 
\end{thm}

By Theorem 1.2 (1) the assumption that the set of all non-isolated points 
of $X$ is compact cannot be omitted. On the other hand, it was unknown 
whether local compactness is necessary. In \cite{Y4}, 
the following question is posed. 
\begin{ques}
Let $J(\aleph_0)$ be the hedgehog space of countable spininess such that 
each spine is a sequence which converges to the center point. Then 
$J(\aleph_0)$ is separable metrizable, but it is not locally compact. It is also 
known that $F_3\bigl(J(\aleph_0)\bigr)$ is metrizable \cite[Theorem 4.11]{Y2} and 
by Theorem 1.2 (2) $F_5\bigl(J(\aleph_0)\bigr)$ is not Fr\'echet-Urysohn. 
Is $F_4\bigl(J(\aleph_0)\bigr)$ Fr\'echet-Urysohn?
\end{ques}
In this paper, we shall give an  affirmative answer to Conjecture 1, and hence Question 1. 
Consequently, we complete the list of equivalent conditions on a metrizable space $X$ 
for $F_n(X)$ to be Fr\'echet-Urysohn for $n\in {\mathbb N}$. 
\par
All topological spaces are assumed to be Tychonoff. By $\mathbb N$ we 
denote the set of all positive natural numbers. Our terminology and 
notation follows \cite{E}. We refer to \cite{AT} and \cite{HR} for 
properties of topological groups and of free topological groups. 
%
%
%
\section{Neighborhoods of the unit element}
In \cite{Y2} we defined a subset $W_n(U)$ of $F_{2n}(X)$ as follows. Let 
${\mathcal U}_{\!\widetilde{X}}$ be the universal uniformity 
on $\widetilde{X}=X\oplus\{e\}\oplus X^{-1}$. For each 
$n\in{\mathbb N}$ and $U\in{\mathcal U}_{\!\widetilde{X}}$, put 
\vspace{5pt}
\[
W_n(U)=\{e\}
\cup\left\{
\begin{array}{lll}
\multicolumn{3}{l}{g=x_1x_2\cdots x_{2k}\in F(X):}  \\
&(1)&\text{$x_i\in\widetilde{X}$ for each $i=1,2,\ldots,2k$ and 
$k\leq n$},  \\
&(2)&\text{$x_1x_2\cdots x_{2k}$ is the reduced form of $g$}, \\
&(3)&\{1,2,\ldots,2k\}=\{i_1,\ldots,i_k\}\oplus\{j_1,\ldots,j_k\}, \\
&(4)& (x_{i_s},x_{j_s}^{\,-1})\in U\text{ for each $s=1,\ldots,k$},\\
&(5)&i_1<i_2<\cdots<i_k, \\
&(6)&\text{$i_s<j_s$ for each $s=1,2,\ldots,k$, and } \\
&(7)&\text{$i_s<i_t<j_s$ iff $i_s<j_t<j_s$ for each $s,t=1,\ldots,k$}. 
\end{array}\right\}.
\]
\par
\vspace{10pt}
\noindent
In the same paper we showed that $W_n(U)$ is a neighborhood of $e$ 
in $F_{2n}(X)$ for each $n\in{\mathbb N}$ and furthermore 
$W(U)=\bigcup\limits_{n=1}^{\infty}W_n(U)$ is a neighborhood of 
$e$ in $F(X)$. In this paper, we apply this for $n=2$ and $n=3$. 
\begin{cor}
Let $U\in{\mathcal U}_{\!\widetilde{X}}$. Then, 
\begin{enumerate}[leftmargin=*, label={\em (\arabic*)}, parsep=0pt, topsep=5pt]
\item $g\in W_2(U)$ if and only if there are 
elements $a,b,c,d$ in $\widetilde{X}$ such that 
$g=abcd$ and $(a,b^{-1})$, $(c,d^{-1})\in U$ or 
$(a,d^{-1})$, $(b,c^{-1})\in U$. 
\item $g\in W_3(U)$ if and only if there are elements 
$p,q,r,s,t,u$ in $\widetilde{X}$ such that $g=pqrstu$ 
and $(p,q^{-1})$, $(r,s^{-1})$, $(t,u^{-1})\in U$, 
$(p,q^{-1})$, $(r,u^{-1})$, $(s,t^{-1})\in U$, $(p,s^{-1})$, 
$(q,r^{-1})$, $(t,u^{-1})\in U$, $(p,u^{-1})$, $(q,r^{-1})$, 
$(s,t^{-1})\in U$, or $(p,u^{-1})$, $(q,t^{-1})$, $(r,s^{-1})\in U$. 
\end{enumerate}
\end{cor}
It is known that $\{W_n(U): U\in{\mathcal U}_{\widetilde{X}}\}$ cannot 
be a neighborhood base of $e$ in $F_{2n}(X)$ for each $n\geq2$. 
On the other hand Uspenski\u\i\ \cite{U} gave the following neighborhood 
base of $e$. 
\par
For each $g\in F(X)$, let $g=x_1^{\varepsilon_1}x_2^{\varepsilon_2}\cdots x_n^{\varepsilon_n}$ be the 
reduced form of $g$, where $x_i\in X$ and $\varepsilon_i=\pm1$ for $i=1,2,\ldots,n$. 
Let $\ell_+(g)=\bigl|\{i\leq n: \varepsilon_i=1\}\bigr|$, $\ell_-(g)=\bigl|\{i\leq n:\varepsilon_i=-1\}\bigr|$ 
and $\ell(g)=\ell_+(g)+\ell_-(g)$. 
Put $F_0=\{g\in F(X): \ell_+(g)=\ell_-(g)\}$. Then $F_0$ is a clopen subgroup 
of $F(X)$. 
Every $h\in F_0$ can be represented as 
\[
h=g_1x_1^{\varepsilon_1}y_1^{-\varepsilon_1}g_1^{-1}
g_2x_2^{\varepsilon_2}y_2^{-\varepsilon_2}g_2^{-1}
\cdots g_nx_n^{\varepsilon_n}y_n^{-\varepsilon_n}g_n^{-1}
\]
for some $n\in{\mathbb N}$, where $x_i, y_i\in X$, $\varepsilon_i
=\pm1$ and $g_i\in F(X)$ for $i=1,2,\ldots,n$. Take an arbitrary 
$r=\{\rho_g:g\in F(X)\}\in P(X)^{F(X)}$. 
Let
\[
p_r(h)=\inf\Bigl\{\sum_{i=1}^n\rho_{g_i}(x_i,y_i):
h=g_1x_1^{\varepsilon_1}y_1^{-\varepsilon_1}g_1^{-1}
\cdots g_nx_n^{\varepsilon_n}y_n^{-\varepsilon_n}g_n^{-1}, 
n\in{\mathbb N}\Bigr\}
\]
for each $h\in F_0$. Then Uspenski\u\i\ proved the following. 
\begin{thm}[\cite{U}] 
\begin{enumerate}[leftmargin=*, label={\em (\arabic*)}, parsep=0pt]
\item $p_r$ is a continuous seminorm on $F_0$ and 
\item $\bigl\{\{h\in F_0:p_r(h)<\delta\}: r\in P(X)^{F(X)}, 
\delta>0\bigr\}$ is a neighborhood base of $e$ in $F(X)$. 
\big(Note that $p_r(e)=0$ for each $r\in P(X)^{F(X)}$.\big)
\end{enumerate}
\end{thm}
In this paper, we apply the above neighborhood base as follows. 
\begin{cor}
Let $\{h_n\}_{n\in{\mathbb N}}$ be a sequence in $F_0$. Then 
the sequence $\{h_n\}_{n\in{\mathbb N}}$ converges to $e$ if and 
only if for each $r\in P(X)^{F(X)}$ and $\delta>0$, there is 
$N\in{\mathbb N}$ such that for each $n\geq N$, there are $m_n\in {\mathbb N}$, 
$x_{n,i}, y_{n,i}\in X$, $g_{n,i}\in F(X)$ and $\varepsilon_{n,i}=\pm1$ for $i=1,2,\ldots,m_n$  
such that 
\[
h_n=g_{n,1}x_{n,1}^{\varepsilon_{n,1}}y_{n,1}^{-\varepsilon_{n,1}}g_{n,1}^{-1}
g_{n,2}x_{n,2}^{\varepsilon_{n,2}}y_{n,2}^{-\varepsilon_{n,2}}g_{n,2}^{-1}
\cdots g_{n,m_n}x_{n,m_n}^{\varepsilon_{n,m_n}}y_{n,m_n}^{-\varepsilon_{n,m_n}}
g_{n,m_n}^{-1} \text{ and } 
\sum\limits_{i=1}^{m_n}\rho_{g_{n,i}}(x_{n,i},y_{n,i})<\delta.
\]
\end{cor}
In the Abelian case, we constructed a neighborhood 
base of $0$ in $A_{2n}(X)$ for each $n\in{\mathbb N}$ \cite{Y1}. In the next section, we use 
the following lemma which is an application of the neighborhood base. For each 
$n\in{\mathbb N}$ define a mapping $j_n:X^n\times X^n\to A_{2n}(X)$ by 
\[
j_n\bigl((x_1,x_2,\ldots,x_n),(y_1,y_2,\ldots,y_n)\bigr)=x_1+x_2+\cdots x_n
-(y_1+y_2+\cdots+y_n)
\]
for each $(x_1,x_2,\ldots,x_n), (y_1,y_2,\ldots,y_n)\in X^n$. 
\begin{lem}[{\cite[Corollary 2.5]{Y1}}]
Let $E$ be a subset of $A_{2n}(X)$ for $n\in{\mathbb N}$. Then, $0\in\overline{E}$ in 
$A_{2n}(X)$ if and only if $j_n^{-1}(E)\cap U^n\neq\emptyset$ for each 
$U\in{\mathcal U}\!\!_X$, where ${\mathcal U}\!\!_X$ is the universal uniformity of $X$ and 
$U^n=\bigl\{\bigl((x_1,x_2,\ldots,x_n),(y_1,y_2,\ldots y_n)\bigr)\in X^n\times X^n: 
(x_i, y_i)\in U, i=1,2,\ldots,n  \bigr\}$. 
\end{lem}
%
%
%

\section{Results}
Let $X$ be a metrizable space and the set $C$ of all non-isolated points of $X$ is compact. 
Since $C$ is compact and each $x\in X\setminus C$ is isolated in $X$, we can take a countable 
neighborhood base $\{W_n: n\in {\mathbb N}\}$ of  $C$ such that  each $W_n$ is clopen in $X$ and 
$W_{n+1}\subseteq W_n$. Let $\{V_n:n\in {\mathbb N}\}$ be a neighborhood 
base of $\varDelta_C$ in $X^2$ such that $V_n\subseteq W_n\times W_n$  
and $\{(y,x)\in X\times X: (x,y)\in V_n\}=V_n$ for each $n\in {\mathbb N}$, where 
$\varDelta_Y=\{(x,x)\in X\times X: x\in Y\}$ for a subset $Y$ of $X$. For each $n\in {\mathbb N}$, let 
\[
U_n= V_n\cup\{(x^{-1}, y^{-1})\in X^{-1}\times X^{-1}: (x,y)\in V_n\}\cup\varDelta_{\widetilde{X}}, 
\]
and ${\mathcal U}=\{U_n:n\in {\mathbb N}\}$. Then $\mathcal U$ is a neighborhood base of 
$\varDelta_{\widetilde{X}}$ in $\widetilde {X}\,\rule{0pt}{12pt}^2$, and hence a base 
for the universal uniformity of $\widetilde{X}\,\rule{0pt}{12pt}^{2}$ such that 
$U_{n+1}\subseteq U_n$ and $\{(y,x):(x,y)\in U_n\}=U_n$ for each $n\in {\mathbb N}$. 
\begin{fact}
For all continuous pseudometric $\rho$ on $X$ and all $\varepsilon>0$, there exists $n\in {\mathbb N}$ 
such that $V_n\subseteq\{(x,y)\in X\times X: \rho(x,y)<\varepsilon\}$. 
\end{fact}
We well begin by showing the following lemma. 
\begin{lem}
Fix $n\in {\mathbb N}$. Let $E\subseteq(X\setminus W_n)\cup(X\setminus W_n)^{-1}$ and 
$F_a\subseteq (W_n\times W_n)$ for each $a\in E$. 
Put $B=\{ax^{\varepsilon}y^{-\varepsilon}a^{-1}\in F_4(X)\setminus F_3(X): (x,y)\in F_a, a\in E, \varepsilon=\pm1\}$. 
If for each $a\in E$ there is $n_a\in {\mathbb N}$ such that $F_a\cap V_{n_a}=\emptyset$, then $e\not\in\overline{B}$. 
\end{lem}
\begin{proof}
Put $Z=W_n$ and $D=X\setminus W_n$. Then $X=Z\oplus D$ and $D$ is a discrete space, 
because $W_n$ is a clopen subset of $X$ and each $d\in D$ is isolated in $X$. 
To show that $e\not\in\overline{B}$, we shall construct some special mappings and 
topological groups defined by Pestov and the author \cite{PY}. 
\par
Define a mapping $\tau: F(D)\times\bigl(Z\times F(D)\bigr)\to Z\times F(D)$ by 
$\tau\bigl(\bigl(g,(x,h)\bigr)\bigr)=(x,gh)$ for each $(x,gh)\in 
F(D)\times\bigl(Z\times F(D)\bigr)$. It is easy to see that $\tau$ is a 
continuous action of the discrete topological group $F(D)$ on the space $Z\times F(D)$. 
For each $g\in F(D)$, the self-homeomorphism $\tau_g:Z\times F(D)\to Z\times F(D);(x,h)
\mapsto(x,gh)$ extends to an automorphism 
$\widetilde{\tau_g}:A\bigl(Z\times F(D)\bigr)\to A\bigl(Z\times F(D)\bigr)$. Since 
$F(D)$ is discrete, $\tau$ gives rise to a continuous action by topological group 
automorphism 
\[
\widetilde{\tau}:F(D)\times A\bigl(Z\times F(D)\bigr)\to A(Z\times F(D)); 
(g,h)\mapsto \widetilde{\tau_g}(h).
\]
\par
Let $G=F(D)\ltimes_{\tau} A\bigl(Z\times F(D)\bigr)$ be the semidirect product 
formed with respect to the action $\widetilde{\tau}$. In other words, as a topological 
space, $G$ is the product of $F(D)$ and $A\bigl(Z\times F(D)\bigr)$ and the group 
operation is given by $(g,a)\cdot(h,b)= \bigl(gh, a+\widetilde{\tau_g}(b)\bigr)$, where 
$g,h\in F(D)$ and $a,b\in A\bigl(Z\times F(D)\bigr)$. Define a mapping 
$\psi: X (=Z\oplus D)\to G$ by 
\[
\psi(t)=\begin{cases}
           \bigl(e,(t,e)\bigr) & \text{if $t\in Z$}, \\
           (t,0) & \text{if $t\in D$},
           \end{cases}
\]
where $e$ and $0$ denote the unit elements of $F(D)$ and $A\bigl(Z\times F(D)\bigr)$ 
respectively. It is easy to see that the mapping $\psi$ is continuous. Hence $\psi$ extends to a 
continuous homomorphism $\widetilde{\psi}:F(X)\to G$. Let 
\begin{align*}
Y&=Z\times(D\oplus D^{-1})\subseteq Z\times F(D)\subseteq A\bigl(Z\times F(D)\bigr) 
\text{ and } \\
f&=(\pi\circ\widetilde{\psi})|_{(\pi\circ\widetilde{\psi})^{-1}(A(Y))}:
(\pi\circ\widetilde{\psi})^{-1}\bigl(A(Y)\bigr)\to A(Y), 
\end{align*}
where $\pi$ is the projection of $G$ onto $A\bigl(Z\times F(D)\bigr)$. 
Since $A(Y)$ can be considered as a topological subgroup of 
$A\bigl(Z\times F(D)\bigr)$, $f$ can be defined and clearly, it is continuous. 
\par
We shall show that $e\not\in\overline{B}$. Let $(x,g)\in Z\times F(D)$. By the 
definition of the action $\widetilde{\tau}$, 
$\widetilde{\tau_g}\bigl((x,e)\bigr)=\tau_g\bigl((x,e)\bigr)=(x,g)\in 
A\bigl(Z\times F(D)\bigr)$. According to the well-known property of the semidirect 
product, for $d\in D$ and $x,y\in Z$, 
\[
(d,0), \bigl(e,(x,e)\bigr), \bigl(e,(y,e)\bigr)\in G, (d,0)^{-1}=(d^{-1},0), \text{ and }
\bigl(e,(y,e)\bigr)^{-1}=\bigl(e,-(y,e)\bigr), 
\]
and hence 
\begin{align*}
(d,0)\cdot\bigl(e,(x,e)\bigr)\cdot\bigl(e,(y,e)\bigr)^{-1}\cdot(d,0)^{-1}&=
\bigl(d,\widetilde{\tau_d}\bigl((x,e)\bigr)\bigr)\cdot\bigl(e,-(y,e)\bigr)\cdot(d^{-1},0) \\
&=\bigl(d,(x,d)\bigr)\cdot\bigl(d^{-1}-(y,e)+\widetilde{\tau_e}(0)\bigr) \\
&=\bigl(d, (x,d)\bigr)\cdot\bigl(d^{-1},-(y,e)\bigr) \\
&=\bigl(e, (x,d)-\widetilde{\tau_d}\bigl((y,e)\bigr)\bigr) \\
&=\bigl(e, (x,d)-(y,d)\bigr).
\end{align*}
From the definition of $\widetilde{\psi}$,  $\widetilde{\psi}(d)=(d,0)$ and 
$\widetilde{\psi}(t)=\bigl(e,(t,e)\bigr)$ for $d\in D$ 
and $t\in Z$ and therefore the word $g=ax^{\varepsilon}y^{-\varepsilon}a^{-1}\in B$ satisfies 
\[
\widetilde{\psi}(g)=\widetilde{\psi}(a)\widetilde{\psi}(x)^{\varepsilon}\widetilde{\psi}(y)^{-\varepsilon}
\widetilde{\psi}(a)^{-1}=\bigl(e,\varepsilon(x,a)-\varepsilon(y,a)\bigr). 
\]
Since $\varepsilon(x,a)-\varepsilon(y,a)\in A(Y)$, $g\in(\pi\circ\widetilde{\psi})^{-1}\bigl(A(Y)\bigr)$. 
It follows that 
$f(g)=\varepsilon(x,a)-\varepsilon(y,a)$ and $f(B)=
\bigl\{\varepsilon(x,a)-\varepsilon(y,a):(x,y)\in F_a, a\in E, \varepsilon=\pm1\bigr\}$. 
Put for each $d\in D\cup D^{-1}$ 
\begin{align*}
O_d&=\bigl\{\bigl((x,d),(y,d)\bigr):(x,y)\in V_{n_d}\bigr\}\cup
\bigl\{\bigl(-(x,d),-(y,d)\bigr):(x,y)\in V_{n_d}\bigr\} & &\text{if $d\in E$} \\
O_d&=\bigl\{\bigl((x,d),(y,d)\bigr):(x,y)\in Z\times Z\bigr\}\cup
\bigl\{\bigl(-(x,d),-(y,d)\bigr):(x,y)\in Z\times Z\bigr\} & &\text{if $d\not\in E$}
\end{align*}
and $O=\bigcup\{O_d:d\in D\cup D^{-1}\}\cup\bigl\{(0,0)\bigr\}
\subseteq \widetilde{Y}\times\widetilde{Y}$. Then $O$ is a neighborhood of 
$\varDelta_{\widetilde{Y}}$ in $\widetilde{Y}\times\widetilde{Y}$. Pick any 
$h\in f(B)$. Then there are $a\in E$, $(x,y)\in F_a$ and $\varepsilon=\pm1$ such 
that $h=\varepsilon(x,a)-\varepsilon(y,a)$. Since $F_a\cap V_{n_a}=\emptyset$, $(x,y)\not\in V_{n_a}$. 
It follows that $\bigl(\varepsilon(x,a), \varepsilon(y,a)\bigr)\not\in O_a$. Clearly, it is also 
$\bigl(\varepsilon(x,a), \varepsilon(y,a)\bigr)\not\in O_d$ for each $d\in (D\cup D^{-1})\setminus\{a\}$. 
Thus, we have that $\bigl(\varepsilon(x,a), \varepsilon(y,a)\bigr)\not\in O$, and hence 
\[
\bigl\{\bigl(\varepsilon(x,a), \varepsilon(y,a)\bigr)\in\widetilde{Y}\times\widetilde{Y}:
 \varepsilon(x,a)-\varepsilon(y,a)\in f(B)\bigl\}\,\cap\,O=\emptyset. 
\]
By Lemma 2.4, this means that $0\not\in\overline{f(B)}$. Since 
$f(e)=0$ and $f$ is continuous, 
we have $e\not\in\overline{B}$. This completes the proof.
\end{proof}
We shall prove our main theorem. 
\begin{thm}
Let $X$ be a metrizable space. If the set of all non-isolated points of $X$ is compact, 
then $F_4(X)$ is Fr\'echet-Urysohn. 
\end{thm}
\begin{proof}
Let $X$ be a metrizable space and the set $C$ of all non-isolated points of $X$ is compact. 
Let $W_n$, $V_n$ and $U_n$ be the sets defined at the top of this section. 
\par
To show that $F_4(X)$ is Fr\'echet-Urysohn, take a subset $A$ of 
$F_4(X)$ and $g\in F_4(X)$ for which $g\in\overline{A}$. We shall 
show that there is a sequence in $A$ converging to $g$. 
Since $F_3(X)$ is closed in $F_4(X)$ and Fr\'echet-Urysohn by 
Theorem 1.2, we may assume that 
$A\subseteq F_4(X)\setminus F_3(X)$. Since the mapping 
\[
i_n|_{i_n^{-1}(F_n(X)\setminus F_{n-1}(X))}:
i_n^{-1}\bigl(F_n(X)\setminus F_{n-1}(X)\bigr)\to F_n(X)\setminus F_{n-1}(X)
\]
is homeomorphism for each $n\in{\mathbb N}$ \cite{A1,J}, $F_4(X)\setminus F_3(X)$ 
is metrizable.  It is also known that $gF_0$ is a 
clopen neighborhood of $g$ in $F(X)$ and $gF_0\cap\bigl(F_4(X)
\setminus F_3(X)\bigr)=\emptyset$ if 
$g\in\bigl(F_3(X)\setminus F_2(X)\bigr)\cup\bigl(F_1(X)\setminus\{e\}\bigr)$ 
(see the definition of the Uspenski\u\i's neighborhood base of $e$ in \S 2). 
These follow that we may assume that 
$g\in\bigl(F_2(X)\setminus F_1(X)\bigr)\cup\{e\}$. 
We consider the following several cases. 
\par
\vspace{5pt}
\noindent
\textit{Case} 1. $g\in\bigl(F_2(X)\setminus F_1(X)\bigr)$.
\par
Let $g=s^{\delta}t^{\gamma}$, where $s,t\in X$, $\delta$, $\gamma=\pm1$ and  
$s^{\delta}\neq t^{-\gamma}$, and take $n_1\in {\mathbb N}$ such that 
$(s^{\delta}, t^{-\gamma})\not\in U_{n_1}$. Since $g\in\overline{A}$, $e\in\overline{g^{-1}A}\in F_6(X)$, 
and hence $g^{-1}A\cap W_3(U_{n_1+k})\neq\emptyset$ for each $k\in {\mathbb N}$. 
Pick a sequence 
$\{h_k=a_k^{\varepsilon_{k,1}}b_k^{\varepsilon_{k,2}}c_k^{\varepsilon_{k,3}}d_k^{\varepsilon_{k,4}}\}$ 
such that $g^{-1}h_k=t^{-\gamma}s^{-\delta}a_k^{\varepsilon_{k,1}}b_k^{\varepsilon_{k,2}}
c_k^{\varepsilon_{k,3}}d_k^{\varepsilon_{k,4}}\in g^{-1}A\cap W_3(U_{n_1+k})$, 
where $a_k, b_k, c_k, d_k\in X$ and $\varepsilon_{k,i}=\pm1$ 
for $i=1,2,3,4$.  For each $k\in {\mathbb N}$, $(t^{-\gamma},s^{\delta})\not
\in \{(y,x): (x,y)\in U_{n_1+k}\}=U_{n_1+k}$. 
It follows from Corollary 2.1 (2) that for each $k\in{\mathbb N}$ 
\begin{align*}
(1)_k \quad (t^{-\gamma},b_k^{-\varepsilon_{k,2}}), (s^{-\delta}, a_k^{-\varepsilon_{k,1}}), 
(c_k^{\varepsilon_{k,3}}, d_k^{-\varepsilon_{k,4}}) &\in U_{n_1+k}, \\[2pt] 
(2)_k\quad (t^{-\gamma},d_k^{-\varepsilon_{k,4}}), (s^{-\delta},a_k^{-\varepsilon_{k,1}}), 
(b_k^{\varepsilon_{k,2}},c_k^{-\varepsilon_{k,3}}) &\in U_{n_1+k}, \text{ or } \\[2pt]
(3)_k \quad (t^{-\gamma},d_k^{-\varepsilon_{k,4}}), (s^{-\delta},c_k^{-\varepsilon_{k,3}}), 
(a_k^{\varepsilon_{k,1}},b_k^{-\varepsilon_{k,2}}) &\in U_{n_1+k}.
\end{align*}
Note that for $p,q\in X$, $n\in {\mathbb N}$ and $\varepsilon_1,\varepsilon_2=\pm1$, 
\[
\text{if $(p^{\varepsilon_1},q^{\varepsilon_2})\in U_n$ then $\varepsilon_1=\varepsilon_2$, 
and $p=q$ or $(p,q)\in V_n$.} 
\]
To show that there is a subsequence of $\{g^{-1}h_k\}$ converging to $e$, take an arbitrary 
$r=\{\rho_g: g\in F(X)\}\in P(X)^{F(X)}$ and $\alpha>0$. 
Let $N_i=\{k\in {\mathbb N}: \text{ the condition $(i)_k$ is satisfied. }\}$ for $i=1,2,3,4$. Then at least 
one of $N_1$, $N_2$, $N_3$ and $N_4$ is infinite. Suppose that $N_1$ or $N_2$ is infinite. 
By Fact 3.1, we can pick $K\in N_1$ such that 
\[
V_{n_1+K}\subseteq\bigl\{(x,y)\in X\times X: \rho_{t^{-\gamma}}(x,y)<\dfrac{\,\alpha\,}{3} \text{ and } 
\rho_e(x,y)<\dfrac{\,\alpha\,}{3}\bigr\}. 
\]
If $N_1$ is inifinite, then for each $k\in N_1$ with $k\geq K$, 
\begin{multline*}
p_r(g^{-1}h_k)=p_r(t^{-\gamma}s^{-\delta}a_k^{\varepsilon_{k,1}}b_k^{\varepsilon_{k,2}}
c_k^{\varepsilon_{k,3}}d_k^{\varepsilon_{k,4}})=
p_r(t^{-\gamma}s^{-\delta}a_k^{\varepsilon_{k,1}}t^{\gamma}et^{-\gamma}b_
k^{\varepsilon_{k,2}}eec_k^{\varepsilon_{k,3}}d_k^{\varepsilon_{k,4}}e) \\
\leq \rho_{t^{-\gamma}}(s,a_k)+\rho_e(t,b_k)+\rho_e(c_k,d_k)<\alpha. 
\end{multline*}
On the other hand, if $N_2$ is infinite, then for each $k\in N_2$ with $k\geq K$, 
\begin{multline*}
p_r(g^{-1}h_k)=p_r(t^{-\gamma}s^{-\delta}a_k^{\varepsilon_{k,1}}b_k^{\varepsilon_{k,2}}
c_k^{\varepsilon_{k,3}}d_k^{\varepsilon_{k,4}})=p_r(t^{-\gamma}s^{-\delta}
a_k^{\varepsilon_{k,1}}t^{\gamma}t^{-\gamma}b_k^{\varepsilon_{k,2}}c_k^{\varepsilon_{k,3}}
t^{\gamma}et^{-\gamma}d_k^{\varepsilon_{k,4}}e) \\
\leq \rho_{t^{-\gamma}}(s,a_k)+\rho_{t^{-\gamma}}(b_k,c_k)+\rho_e(t,d_k)<\alpha.
\end{multline*}
Suppose that $N_3$ is infinite. Then, by Fact 3.1, we can pick $K\in N_3$ such that 
\[
V_{n_1+K}\subseteq\bigl\{(x,y)\in X\times X: \rho_{t^{-\gamma}s^{-\delta}}(x,y)<\dfrac{\,\alpha\,}{3}, 
\rho_{t^{-\gamma}}(x,y)<\dfrac{\,\alpha\,}{3} \text{ and } 
\rho_e(x,y)<\dfrac{\,\alpha\,}{3}\bigr\}.
\]
Then, for each $k\in N_3$ with $k\geq K$, 
\begin{multline*}
p_r(g^{-1}h_k^{-1})=p_r(t^{-\gamma}s^{-\delta}a_k^{\varepsilon_{k,1}}b_k^{\varepsilon_{k,2}}
c_k^{\varepsilon_{k,3}}d_k^{\varepsilon_{k,4}})=p_r(t^{-\gamma}s^{-\delta}
a_k^{\varepsilon_{k,1}}b_k^{\varepsilon_{k,2}}s^{\delta}t^{\gamma}t^{-\gamma}s^{-\delta}
c_k^{\varepsilon_{k,3}}t^{\gamma}et^{-\gamma}d_k^{\varepsilon_{k,4}}e) \\
\leq\rho_{t^{-\gamma}s^{-\delta}}(a_k,b_k)+\rho_{t^{-\gamma}}(s,c_k)+
\rho_e(t,d_k)<\alpha.
\end{multline*}
Therfore, by Corollary 2.3, we can find a subsequence $\{g^{-1}h_{n_k}\}$ of $\{g^{-1}h_k\}$ 
converges to $e$. It follows that the sequence $\{h_{n_k}\}$ in $A$ converges to $g$. 
\par
\vspace{5pt}
\noindent
\textit{Case} 2. $g=e$.
\par
Since $e\in\overline{A}$, $A\cap W_2(U_n)\neq\emptyset$ for each $n\in {\mathbb N}$. Let 
\begin{align*}
H(U_n)&=\{abcd\in W_2(U_n): (a,b^{-1}), (c,d^{-1})\in U_n\}, \\
I(U_n)&=\{abcd\in W_2(U_n): (a,d^{-1}), (b,c^{-1})\in U_n, a\neq d^{-1}\} \text{ and } \\
J(U_n)&=\{abcd\in W_2(U_n): (a, d^{-1}), (b,c^{-1})\in U_n, a=d^{-1}\} \\
&=\{abca^{-1}\in W_2(U_n): (b, c^{-1})\in U_n\}.
\end{align*}
By Corollary 2.1 (1), $W_2(U_n)= H(U_n)\cup I(U_n)\cup J(U_n)$. 
We discuss in more three cases. 
\par
\vspace{5pt}
\noindent
\textit{Case} 2-1. \textit{$M_1=\{n\in {\mathbb N}: A\cap H(U_n)\neq\emptyset\}$ is infinite.}
\par
For each $n\in M_1$, pick $g_n=a_n^{\gamma_n}b_n^{\delta_n}c_n^{\varepsilon_n}
d_n^{\zeta_n}\in A\cap H(U_n)$, where $a_n, b_n, c_n, d_n\in X$ and 
$\gamma_n, \delta_n, \varepsilon_n, \zeta_n=\pm1$. For each $n\in M_1$, 
since $\ell(g_n)=4$ and $(a_n^{\gamma_n}, b_n^{-\delta_n}), (c_n^{\varepsilon_n}, d_n^{-\zeta_n})\in U_n$,   
$(a_n, b_n)$, $(c_n, d_n)\in V_n$. Thus, by the same argument as Case 1, we can show that 
the sequence $\{g_n\}$ in $A$ converges to $e$. 
\par
\vspace{5pt}
\noindent
\textit{Case} 2-2. \textit{$M_2=\{n\in {\mathbb N}: A\cap I(U_n)\neq\emptyset\}$ is infinite.}
\par
For each $n\in M_2$, pick 
$g_n=a_n^{\gamma_n}b_n^{\delta_n}c_n^{\varepsilon_n}d_n^{\zeta_n}\in A\cap I(U_n)$, 
where $a_n, b_n, c_n, d_n\in X$ and $\gamma_n, \delta_n, \varepsilon_n, \zeta_n=\pm1$. 
In this case, it is easy to see that $(a_n, d_n), (b_n,c_n)\in V_n$ for each $n\in M_2$. 
Since $\{V_n: n\in {\mathbb N}\}$ is a neighborhood base of the compact set 
$\varDelta_C$ in $X\times X$, there are an infinite subset $M'_2$ of $M_2$ and 
$c\in C$ such that both of the sequences $\{a_n\}_{n\in M'_2}$ and $\{d_n\}_{n\in M'_2}$ 
converge to $c$. Let $r=\{\rho_g: g\in F(X)\}\in P(X)^{F(X)}$ and $\alpha>0$. we can choose 
$K\in M'_2$ such that 
\[
V_K\subseteq\bigl\{(x,y)\in X\times X: \rho_c(x,y)<\dfrac{\,\alpha\,}{3}, \rho_{c^{-1}}(x,y)
<\dfrac{\,\alpha\,}{3} \text{ and } 
\rho_e(x,y)<\dfrac{\,\alpha\,}{3}\bigr\}.
\]
Then, for each $n\in M'_2$ with $n\geq K$, 
\begin{multline*}
p_r(g_n)=p_r(a_n^{\gamma_n}b_n^{\delta_n}c_n^{\varepsilon_n}d_n^{\zeta_n})
=p_r(ea_n^{\gamma_n}c^{-\gamma_n}ec^{\gamma_n}b_n^{\delta_n}c_n^{\varepsilon_n}
c^{-\gamma_n}ec^{\gamma_n}d_n^{\zeta_n}e) \\
\leq\rho_e(a_n,c)+\rho_{c^{\gamma_n}}(b_n, c_n)+\rho_e(c, d_n)<\alpha. 
\end{multline*}
Thus, the sequence $\{g_n\}_{n\in M'_2}$ in $A$ converges to $e$. 
\par
\vspace{5pt}
\noindent
\textit{Case} 2-3. \textit{Case 2, but neither Case 2-1 nor Case 2-2.}
\par
In this case, for convenience, we may assume that for each $n\in {\mathbb N}$ 
$A\cap\bigl(H(U_n)\cup I(U_n)\bigr)=\emptyset$. It follows that 
$A\cap J(U_n)=A\cap W_2(U_n)\neq\emptyset$ for each $n\in {\mathbb N}$. Since $W_2(U_n)$ is a 
neighborhood of $e$, we have that $e\in\overline{A\cap J(U_n)}$ for each $n\in {\mathbb N}$. 
Put 
\[
J_1(U_n)=\{abca^{-1}\in J(U_n): a\in C\cup C^{-1}\} \text{ and }
J_2(U_n)=\{abca^{-1}\in J(U_n): a\not\in C\cup C^{-1}\}
\]
 for each $n\in {\mathbb N}$. Clearly $J(U_n)=J_1(U_n)\cup J_2(U_n)$, 
 and hence $A\cap J_1(U_n)\neq\emptyset$ or $A\cap J_2(U_n)\neq\emptyset$. 
 Suppose that $L_1=\{n\in {\mathbb N}: A\cap J_1(U_n)\neq\emptyset\}$ is infinite. 
Take $g_n=a_n^{\gamma_n}b_n^{\delta_n}c_n^{\varepsilon_n}a_n^{-\gamma_n}\in A\cap J_2(U_n)$, where 
$a_n, b_n, c_n\in X$ and $\gamma_n, \delta_n, \varepsilon_n=\pm1$ for each $n\in L_1$. 
Since $\ell(g_n)=4$, $(b_n, c_n)\in V_n$, and also $a_n$ is in the compact set $C$. 
We can take an infinite subset $L'_1$ of $L_1$ such that the sequence $\{a_n\}_{n\in L'_1}$ 
converges to a point of $C$. Hence, with the same argument as Case 2-2, we can show 
that the sequence $\{g_n\}_{n\in L'_1}$ in $A$ converges to $e$. 
\par
Therefore, we may assume that $A\cap J_2(U_n)=A\cap W_2(U_n)\neq\emptyset$ for each 
$n\in {\mathbb N}$. Let $B=A\cap J_2(U_1)$. Then $e\in\overline{B}$. For each $g\in B$, let 
$g=a_g^{\gamma_g}b_g^{\delta_g}c_g^{\varepsilon_g}a_g^{-\gamma_g}$, where $a_g, b_g, c_g\in X$ and 
$\gamma_g, \delta_g, \varepsilon_g=\pm1$. Note that $(b_g, c_g)\in V_1\subseteq X\times X$ 
and $a_g\not\in C$. Put $E=\{a_g\in X: g\in B\}$ and for each $a\in E$, 
$F_a=\{(b_g, c_g)\in X\times X: a=a_g \text{ for some } g\in B\}$. 
Suppose that  there are $a\in E$ and an infinite subset $L_2$ of 
${\mathbb N}$ such that $F_a\cap V_n\neq\emptyset$ 
for each $n\in L_2$. Pick $(b_{g_n}, c_{g_n})\in F_a\cap V_n$ for each $n\in L_2$.  Since 
$a_{g_n}=a$ and $\gamma_{g_n}=\pm1$ for each $n\in L_2$, it is easy to check that the sequence 
$\bigl\{a_{g_n}^{\gamma_{g_n}}b_{g_n}^{\delta_{g_n}}c_{g_n}^{\varepsilon_{g_n}}
a_{g_n}^{-\gamma_{g_n}}\bigr\}=\bigl\{a^{\gamma_{g_n}}b_{g_n}^{\delta_{g_n}}
c_{g_n}^{\varepsilon_{g_n}}a^{-\gamma_{g_n}}\bigr\}$ in $B$ converges to $e$. 
Thus we may assume that 
\[
\text{for each $a\in E$ there is $n_a\in {\mathbb N}$ such that 
$F_a\cap V_{n_a}=\emptyset$.}\quad (\ast)
\]
Pick $g\in B$. Then $(b_g, c_g)\in V_1$ and $a_g\not\in C$. Since $C=\bigcap\limits_{n=1}^{\infty}W_n$, 
choose $n_1\in {\mathbb N}$ such that $a_g\not\in W_{n_1}$. Note that 
$B=\{g\in B:a_g\in W_{n_1}\}\cup\{g\in B:a_g\not\in W_{n_1}\}$.  It follows from the assumption 
$(\ast)$ and Lemma 3.2 that $e\not\in\overline{\{g\in B: a_g\not\in W_{n_1}\}}$. Put 
$B_1=\{g\in B:a_g\in W_{n_1}\}$. Since $e\in\overline{B}$, $e\in\overline{B_1}$, and hence 
$B_1\cap J_2(U_{n_1})\neq\emptyset$. It follows that we can 
pick $g_{n_1}\in B_1\cap J_2(U_{n_1})$. Then $(b_{g_{n_1}}, c_{g_{n_1}})\in V_{n_1}$ 
and $a_{g_{n_1}}\in W_{n_1}\setminus C$. Choose $n_2\in {\mathbb N}$ with 
$a_{g_{n_1}}\not\in W_{n_2}$, so that $n_2>n_1$. 
Since $B_1=\{g\in B: a_g\in W_{n_2}\}\cup\{g\in B: a_g\in W_{n_1}\setminus W_{n_2}\}$ and 
$e\not\in\overline{\{g\in B: a_g\in W_{n_1}\setminus W_{n_2}\}}$ by the assumtion $(\ast)$ 
and Lemma 3.2, we can pick $g_{n_2}\in B_2\cap J_2(U_{n_2})$, where $B_2=\{g\in B: a_g\in W_{n_2}\}$. 
In a similar way, we obtain a subsequence $\{n_k\}$ of $ {\mathbb N}$ and a sequence $\{g_{n_k}\}$ 
in B such that $g_{n_k}\in B_k\cap J_2(U_{n_k})$, where $B_k=\{g\in B: a_g\in W_{n_k}\}$ 
for each $k\in {\mathbb N}$. Since for each $k\in {\mathbb N}$ 
\[
g_{n_k}=a_{g_{n_k}}^{\gamma_{g_{n_k}}}b_{g_{n_k}}^{\delta_{g_{n_k}}}c_{g_{n_k}}^{\varepsilon_{g_{n_k}}}
a_{g_{n_k}}^{-\gamma_{g_{n_k}}}, \ (b_{g_{n_k}}, c_{g_{n_k}})\in V_{n_k}
\text{ and } a_{g_{n_k}}\in W_{n_k}. 
\]
Note that the sequence $\{a_{g_{n_k}}\}$ converges to an element of $C$. 
Thus, with the same aurgument as Case 2.2, we can show that $\{g_{n_k}\}$ converges to $e$. 
\par
In all cases, we can find a sequence in $A$ converging to $g$. Therefore, we conclude that 
$F_4(X)$ is Fr\'echet-Urysohn. 
\end{proof}
We give an affirmative answer to Conjecture 1 and also Question 1, as follows.
\begin{cor}
Let $X$ be a metrizable space. Then{\em :}
\begin{enumerate}[leftmargin=*,label={\em (\arabic*)},parsep=0pt, topsep=5pt] 
\item The following are equivalent{\em:} 
\begin{enumerate}[leftmargin=*,label={\em (\roman*)},parsep=0pt, topsep=5pt]
\item $F_n(X)$ is Fr\'echet-Urysohn for each $n\in {\mathbb N}${\em;}
\item $F_5(X)$ is Fr\'echet-Urysohn{\em;}
\item $X$ is compact or discrete.
\end{enumerate}
\item The following are equivalent{\em:}
\begin{enumerate}[leftmargin=*,label={\em (\roman*)},parsep=0pt, topsep=5pt]
\item $A_n(X)$ is Fr\'echet-Urysohn for each $n\in {\mathbb N}${\em;}
\item $A_3(X)$ is Fr\'echet-Urysohn{\em;}
\item $F_4(X)$ is Fr\'echet-Urysohn{\em;}
\item $F_3(X)$ is Fr\'echet-Urysohn{\em;} 
\item the set of all non-isolated points of $X$ is compact. 
\end{enumerate}
\item Both $A_2(X)$ and $F_2(X)$ are Fr\'echet-Urysohn. 
\end{enumerate} 

\end{cor}
\begin{cor}
Let $J(\kappa)$ be the hedgehog space of $\kappa$ many spininess such that 
each spine is a sequence which converges to the center point. Then 
$F_4\bigl(J(\kappa)\bigr)$ is Fr\'echet-Urysohn. 
\end{cor}
%
%
%
%
%
\vspace{10pt}
\noindent
\textbf{References}
\par


\vspace{6pt}
\begin{thebibliography}{99}
%
\bibitem{A1}{A.~V.~Arhangel'ski\u\i}, \textit{Mapping related to topological groups}, 
Soviet Math. Dokl. \textbf{9} (1968) 1011-1015.
%
\bibitem{AOP}{A.~V.~Arhangel'ski\u\i, O.~G.~Okunev and V.~G.~Pestov},
    \textit{Free topological groups over metrizable spaces}, Topology
    Appl. \textbf{33} (1989) 63-76.
%
\bibitem{AT}{A.~V.~Arhangel'ski\u\i \/and M.~Tkachenko},
    \textit{Topological Group and Related Structures},
    Atlantis Press and World Sci., Paris, 2008.
%
\bibitem{E}{R.~Engelking}, \textit{General Topology}
    (Heldermann, Berlin, 1989).
%
%
\bibitem{HR}{E.~Hewitt and K.~Ross}, \textit{Abstract harmonic
    analysis I}, Academic Press, New York, (1963).
%
\bibitem{J}{C.~Joiner}, \textit{Free topological groups and dimension}, 
Trans. Amer. Math. Soc. \textbf{220} (1976) 401-418. 
%
%
%
\bibitem{M}{A.~A.~Markov}, \textit{On free topological groups}, Izv.
    Akad. Nauk SSSR Ser. Mat. \textbf{9} (1945) 3-64 (in Russian);
    Amer. Math. Soc. Transl. \textbf{8} (1962) 195-272.
%
\bibitem{PY}{V.~Pestov and K.~Yamada}, \textit{Free topological groups on 
    metrizable spaces and inductive limits}, Topology Appl. \textbf{98} (1999) 
    291-301.

%
\bibitem{U}{V.~V.~Uspenski\u\i}, \textit{Free topological groups of
    metrizable spaces}, Math. USSR Izvestiya \textbf{37} (1991)
    657-680.
%
\bibitem{Y1}{K.~Yamada}, \textit{Characterizations of a metrizable
    space $X$ such that every $A_n(X)$ is a $k$-space},
    Topology Appl. \textbf{49} (1993) 75-94.
%
\bibitem{Y2}{K.~Yamada}, \textit{Metrizable subspaces of free
    topological groups on metrizable spaces}, Topology Proc.
    \textbf{23} (1998) 379-409.
%
\bibitem{Y3}{K.~Yamada}, \textit{Fr\'echet-Urysohn spaces in free
    topological groups on metrizable spaces}, Proc. Amer. Math.
    Soc., \textbf{130} (2002) 2461-2469.
%
%
\bibitem{Y4}{K.~Yamada}, \textit{Fr\'echet-Urysohn subspaces of free topological groups}, 
    Topology Appl. \textbf{210} (2016), 81-89.
%
\bibitem{Y5}{K.~Yamada}, \textit{Fr\'echet-Urysohn subspaces of free topological groups II}, 
    submitted. 
\end{thebibliography}
\end{document}